\documentclass[11pt]{amsart}

\usepackage[a4paper,margin=30mm]{geometry}
\usepackage[T1]{fontenc}
\usepackage{lmodern}
\usepackage{textcomp}
\usepackage{amsmath,amssymb,amsthm,mathtools}
\usepackage{aliascnt}
\usepackage{booktabs,tabularx,array}
\usepackage{enumitem}
\usepackage{microtype}
\usepackage{tikz}
\usepackage[hidelinks]{hyperref}
\hypersetup{
 pdftitle={The Multiset Dimension of Graphs: Extremal Values and King Grids},
 pdfauthor={Jaan Allikvere},
 pdfkeywords={multiset dimension, multiset resolving set, distance degree sequence, counterexample, king grid, strong product, Chebyshev metric}
}
\usepackage[nameinlink,capitalise,noabbrev]{cleveref}
\crefname{theorem}{Theorem}{Theorems}
\crefname{lemma}{Lemma}{Lemmas}
\crefname{proposition}{Proposition}{Propositions}
\crefname{corollary}{Corollary}{Corollaries}
\crefname{remark}{Remark}{Remarks}
\crefname{observation}{Observation}{Observations}
\crefname{problem}{Problem}{Problems}

\newtheorem{theorem}{Theorem}[section]
\newaliascnt{lemma}{theorem}
\newtheorem{lemma}[lemma]{Lemma}
\aliascntresetthe{lemma}
\newaliascnt{proposition}{theorem}
\newtheorem{proposition}[proposition]{Proposition}
\aliascntresetthe{proposition}
\newaliascnt{corollary}{theorem}
\newtheorem{corollary}[corollary]{Corollary}
\aliascntresetthe{corollary}
\newaliascnt{observation}{theorem}
\newtheorem{observation}[observation]{Observation}
\aliascntresetthe{observation}
\newaliascnt{problem}{theorem}
\newtheorem{problem}[problem]{Problem}
\aliascntresetthe{problem}
\theoremstyle{remark}
\newaliascnt{remark}{theorem}
\newtheorem{remark}[remark]{Remark}
\aliascntresetthe{remark}

\newcommand{\mdim}{\operatorname{dim}_{\mathrm m}}
\newcommand{\dist}{d}
\newcommand{\ms}[1]{\mathopen{\{\!\{}#1\mathclose{\}\!\}}}
\newcommand{\King}{\mathcal K}
\newcommand{\D}{\mathcal D}
\newcommand{\partheading}[1]{\bigskip\begin{center}{\Large\bfseries #1}\end{center}\medskip}

\title[Multiset dimension: extremal values and king grids]{The Multiset Dimension of Graphs:\\ Extremal Values and King Grids}
\author{Jaan Allikvere}
\thanks{Independent researcher, Tallinn, Estonia.  \emph{E-mail}: \texttt{jaan@lahendus.ee}.  ORCID: 0009-0003-5228-7015.}
\date{July 30, 2026}

\begin{document}

\begin{abstract}
We present three results on the multiset dimension of graphs, resolving one conjecture and two open questions from the literature.  First, we disprove the conjecture of Simanjuntak, Siagian and Vetr\'ik (2017) that every graph $G$ of order $n(G)$ with finite multiset dimension satisfies $\mdim(G)\le n(G)-1$: an exhaustive computation over all $1{,}018{,}690{,}328$ connected graphs of orders $2$ through $11$ shows that exactly eight graphs attain $\mdim(G)=n(G)$, all of order $11$, so $11$ is the smallest order at which the trivial upper bound is attained.  This also answers a question from the recent survey of Farhan, Klav\v{z}ar, Kuziak and Yero.  Second, we prove that $\mdim(P_n\boxtimes P_n)=4$ for every $n\ge 5$, answering a question of Hakanen and Yero: after a $45^\circ$ change of coordinates the Chebyshev metric of the king grid becomes half the Manhattan metric on a parity sublattice, and four boundary inequalities reduce every potentially resolving three-landmark set to two geometric cases, in each of which we exhibit an explicit collision.  Third, on king strips the parameter grows linearly: $\mdim(P_3\boxtimes P_n)=n$ for $n\ge 6$ (with $\mdim(P_3\boxtimes P_3)=\infty$, $\mdim(P_3\boxtimes P_4)=5$, $\mdim(P_3\boxtimes P_5)=6$), where the lower bound rests on three local separation conditions and a finite min-plus transfer certificate whose equality case yields a finite automaton with a $19$-state recurrent core, and the upper bound is an explicit landmark pattern of period three that works for every height.  Combined with a blindness lower bound, $\mdim(P_h\boxtimes P_n)=\Theta(n)$ for every fixed $h\ge 3$, so the constant answer on square king grids requires both dimensions to grow.
\end{abstract}

\maketitle

\section{Introduction}\label{sec:intro}

Let $G$ be a finite, simple, connected graph with vertex set $V(G)$ of order $n(G)$, and let $\dist(u,v)$ denote the usual shortest-path distance.  The classical notion of a resolving set goes back to Slater~\cite{Slater1975} and Harary and Melter~\cite{HararyMelter1976}: an ordered set $W=\{w_1,\dots,w_k\}\subseteq V(G)$ resolves $G$ if the distance vectors $\bigl(\dist(v,w_1),\dots,\dist(v,w_k)\bigr)$ are pairwise distinct, and the metric dimension $\dim(G)$ is the minimum size of such a set.

Simanjuntak, Siagian and Vetr\'ik~\cite{SimanjuntakEtAl} introduced a natural variant in which the landmarks are unlabelled.\footnote{The survey \cite{FarhanEtAl2026} records that the concept was independently introduced earlier by Saenpholphat (2009), in Thai-language conference proceedings, and rediscovered in \cite{SimanjuntakEtAl}; the conjecture below is from \cite{SimanjuntakEtAl}.}  For $S=\{s_1,\dots,s_k\}\subseteq V(G)$, the \emph{multiset representation} of a vertex $x$ with respect to $S$ is
\[
 m_G(x\mid S)=\ms{\dist_G(x,s_1),\dots,\dist_G(x,s_k)},
\]
the multiset of distances from $x$ to the landmarks, without the labels of the landmarks (in particular, if $x\in S$ then $0\in m_G(x\mid S)$).  The set $S$ is a \emph{multiset resolving set} ($m$-resolving set) if $m_G(u\mid S)\ne m_G(v\mid S)$ for all $u\ne v$; the \emph{multiset dimension} $\mdim(G)$ is the minimum cardinality of an $m$-resolving set if one exists, and $\mdim(G)=\infty$ otherwise.  When the landmarks are given an order, we write
\[
 r_G(x\mid s_1,\dots,s_k)=(\dist_G(x,s_1),\dots,\dist_G(x,s_k))
\]
for the corresponding \emph{ordered distance vector}, omitting the subscript $G$ when the ambient graph is clear; two vertices with equal ordered vectors have, in particular, equal multiset representations.  Deciding whether $\mdim(G)\leq k$ is NP-hard \cite{HakanenYero2024}, and the parameter has attracted steady attention; see the two recent surveys \cite{FarhanEtAl2026,AlbejaniEtAl2026}.

This paper collects three results on the parameter, at its two extremes: the maximum possible value of $\mdim$ on general graphs (Part~I), and its exact behaviour on king grids --- strong products of paths --- where it is constant on squares (Part~II) but grows linearly on strips (Part~III).  Each part is self-contained and can be read independently.

\subsection{The trivial upper bound is attained}

Since every $m$-resolving set is contained in $V(G)$, a graph with finite multiset dimension trivially satisfies $\mdim(G)\le n(G)$.  The introducing paper proposed that this trivial bound is never attained.

\begin{quote}
\textbf{Conjecture} (\cite[Conjecture~2.1]{SimanjuntakEtAl}).  \emph{If $G$ is a graph on $n$ vertices having finite multiset dimension, then $\mdim(G)\le n-1$.}
\end{quote}

The conjecture was supported by all known exact values and bounds.  Paths are the only graphs with $\mdim(G)=1$, no graph has $\mdim(G)=2$ \cite{SimanjuntakEtAl}, and for trees of diameter at least two the stronger bound $\mdim(T)\le n-2$ holds whenever $\mdim(T)$ is finite \cite{HafidhEtAl2019}.  The known realizability construction produces, for $k\ge 3$ and $n\ge 3(k-1)$, a graph of order $n$ with $\mdim(G)=k$ (Khemmani and Isariyapalakul \cite[Theorem~5]{KhemmaniIsariyapalakul2018}, as stated in \cite{FarhanEtAl2026}); note that it reaches only $k\le n/3+1$, far below the conjectured threshold, and the top of the range remained unexplored.  The survey \cite{FarhanEtAl2026} states the situation plainly: ``regarding the trivial upper bound $\mdim(G)\le n(G)$, it is not clear whether there is a graph $G$ for which $\mdim(G)=n(G)$.''

Our first result shows that such graphs exist and that the conjecture is false.

\begin{theorem}\label{thm:order}
There are exactly eight connected graphs $G$ with $2\le n(G)\le 11$ and
\[
\mdim(G)=n(G),
\]
and all eight have order exactly $11$.  They are pairwise non-isomorphic, and each has $23$--$26$ edges, diameter $3$, girth $3$, and trivial automorphism group.  In particular, Conjecture~2.1 of \cite{SimanjuntakEtAl} is false, the trivial bound $\mdim(G)\le n(G)$ is best possible, and $11$ is the smallest order of a graph attaining it.
\end{theorem}

We remark that only connected graphs need to be considered: in a disconnected graph two vertices in different components have no finite distance, and the multiset dimension is defined via shortest-path distances, so the parameter is studied for connected graphs \cite{SimanjuntakEtAl,FarhanEtAl2026}.  We also disregard the one-vertex graph $K_1$, for which the value of $\mdim$ is a matter of convention.

\subsection{Square king grids}

Hakanen and Yero~\cite{HakanenYero2024} studied the strong product $P_n\boxtimes P_n$, also known as the $n\times n$ king grid.  They established
\[
 \mdim(P_4\boxtimes P_4)=6,
 \qquad
 \mdim(P_5\boxtimes P_5)=\mdim(P_6\boxtimes P_6)=4,
\]
and proved that
\[
 3\le \mdim(P_n\boxtimes P_n)\le 4 \qquad (n\ge 7).
\]
They explicitly asked whether the upper bound is always attained.  This question was still recorded as open in two independent 2026 surveys~\cite{AlbejaniEtAl2026,FarhanEtAl2026}.  Our second result closes the remaining gap.

\begin{theorem}\label{thm:square}
For every integer $n\ge 5$,
\[
 \mdim(P_n\boxtimes P_n)=4.
\]
\end{theorem}

\subsection{King strips}

The survey \cite{FarhanEtAl2026} raises the general rectangular problem (Problem~1): determine $\mdim(P_{n_1}\boxtimes P_{n_2})$ for all $n_1,n_2\ge 3$.  Our third result solves the height-three case exactly, and reveals that the answer is of a completely different nature from the square case.

\begin{theorem}\label{thm:strip}
For every $n\ge 6$,
\[
 \mdim(P_3\boxtimes P_n)=n .
\]
Moreover $\mdim(P_3\boxtimes P_3)=\infty$, $\mdim(P_3\boxtimes P_4)=5$, and $\mdim(P_3\boxtimes P_5)=6$.
\end{theorem}

Thus a king strip of height three requires one landmark per column, whereas an $n\times n$ king square is resolved by four landmarks.  The blindness mechanism behind the lower bound, combined with a landmark pattern that applies to every height, yields a two-sided linear bound for every fixed height.

\begin{theorem}\label{thm:strips}
For all $h\ge 3$ and $n\ge h$,
\[
 \mdim(P_h\boxtimes P_n)\;\ge\;\left\lceil\frac{n}{2h-3}\right\rceil,
\]
and for all $h\ge 3$ and $n\ge 2h+1$,
\[
 \mdim(P_h\boxtimes P_n)\;\le\;n .
\]
In particular $\mdim(P_h\boxtimes P_n)=\Theta(n)$ for every fixed $h\ge 3$.
\end{theorem}

Together with \cref{thm:square}, this shows that on rectangles $P_{n_1}\boxtimes P_{n_2}$ the multiset dimension can stay bounded only if the aspect ratio stays bounded: for fixed $n_1$ it grows linearly in $n_2$.

\subsection{Organization}

Part~I (\cref{sec:prelim,sec:graphs,sec:search,sec:remarks}) proves \cref{thm:order}: the eight graphs are listed in \cref{sec:graphs} (\cref{tab:eight}), where one of them is analysed by hand; \cref{sec:search} describes the exhaustive computation --- a complete scan of all connected graphs of orders $2$ through $11$, including all $1{,}006{,}700{,}565$ connected graphs of order $11$ --- together with the correctness arguments and independent re-verifications; \cref{sec:remarks} collects consequences and open problems.  Part~II (\cref{sec:rotated,sec:boundary,sec:antipodal,sec:distinct,sec:repeated,sec:proofsquare,sec:rect,sec:verifsquare}) proves \cref{thm:square}.  Part~III (\cref{sec:local,sec:lower,sec:extremal,sec:upper,sec:strips,sec:landscape,sec:verification}) proves \cref{thm:strip,thm:strips} and determines the structure of all candidate minimum resolving sets on the height-three strip.  All code, data, and certificates are permanently archived; see the data availability statement at the end.

\partheading{Part I. The multiset dimension can equal the order}

\section{Preliminaries: why the conjecture was plausible, and why it fails}\label{sec:prelim}

Throughout Part~I we write $r_{\mathrm m}(v\mid W)=m_G(v\mid W)$ for the multiset representation of $v$ with respect to $W\subseteq V(G)$ (\cref{sec:intro}).

The \emph{distance degree sequence} $\mathrm{DDS}(v)$ of a vertex $v$ is the sequence $(a_0,a_1,a_2,\dots)$ where $a_i$ is the number of vertices at distance exactly $i$ from $v$ (so $a_0=1$ and $a_1=\deg v$; including the constant initial entry $a_0=1$ does not affect injectivity).  A graph is \emph{distance degree injective} (DDI) if its vertices have pairwise distinct distance degree sequences; these graphs were introduced by Bloom, Kennedy and Quintas \cite{BloomKennedyQuintas1983}.

\begin{observation}\label{obs:ddi}
For a connected graph $G$, the full vertex set $V(G)$ is an $m$-resolving set if and only if $G$ is DDI.
\end{observation}

\begin{proof}
The multiset $r_{\mathrm m}(v\mid V(G))$ contains the entry $i$ with multiplicity $a_i$ for every $i\ge 0$, where $(a_0,a_1,\dots)=\mathrm{DDS}(v)$.  Hence two vertices have equal representations with respect to $V(G)$ if and only if they have equal distance degree sequences.
\end{proof}

\begin{corollary}\label{cor:char}
$\mdim(G)=n(G)$ if and only if $G$ is DDI and no proper subset of $V(G)$ is $m$-resolving.
\end{corollary}

Note that a DDI graph has trivial automorphism group, since an automorphism preserves distance degree sequences; this explains the corresponding column of \cref{tab:eight}.

At first sight \cref{cor:char} may look self-contradictory: if the full vertex set resolves, should not some $(n-1)$-subset resolve as well?  This intuition is exactly what makes the conjecture plausible --- and it is wrong, because $m$-resolvability is \emph{not monotone} under adding or removing landmarks.  Both directions fail:

\begin{itemize}[leftmargin=2em]
\item \emph{Adding a landmark can destroy resolution.}  In $K_2$ the singleton $\{w\}$ is $m$-resolving (representations $\ms{0}$ and $\ms{1}$), but the full vertex set is not: both vertices receive $\ms{0,1}$.  Thus $\mdim(K_2)=1$ although $V(K_2)$ does not resolve.
\item \emph{Removing a landmark from a resolving full vertex set can leave nothing resolvable at the same size.}  The smallest DDI graph in which \emph{no} $(n-1)$-subset of $V(G)$ is $m$-resolving has order $8$ (graph6 code \texttt{G?\kern0pt\textasciigrave cmW}); yet its multiset dimension is $3$, attained by a much smaller set.  So even in DDI graphs the resolving subsets of $V(G)$ do not form a ``nested'' family.
\end{itemize}

The counterexamples of \cref{thm:order} are the extreme manifestation of this non-monotonicity: graphs in which the full vertex set resolves while \emph{every} proper subset, of every size from $1$ to $n-1$, fails.

\section{The eight graphs}\label{sec:graphs}

\begin{table}[ht]
\centering
\caption{The eight connected graphs of order $11$ with $\mdim(G)=11$ (graph6 codes; also provided in machine-readable form in the ancillary file \texttt{candidates11.txt}).  Every graph has diameter $3$, girth $3$, and trivial automorphism group.}\label{tab:eight}
\small
\begin{tabular}{cllc}
\toprule
 & graph6 code & degree sequence & $|E|$\\
\midrule
$G_1$ & \texttt{J?BDf?[hvq\_}                 & $(2,3,3,3,4,4,4,5,5,6,7)$ & $23$\\
$G_2$ & \texttt{J?\kern0pt\textasciigrave CRJWzNl\_}  & $(2,3,3,3,4,4,5,5,6,7,8)$ & $25$\\
$G_3$ & \texttt{J?AF?zUyfh\_}                 & $(2,3,3,3,4,4,5,5,6,6,7)$ & $24$\\
$G_4$ & \texttt{J?B@eRhuT\char`\\\_}          & $(2,3,3,3,4,4,5,5,6,6,7)$ & $24$\\
$G_5$ & \texttt{J?bB\textasciigrave rcb\textasciicircum h\_} & $(3,3,3,4,4,4,5,5,6,6,7)$ & $25$\\
$G_6$ & \texttt{J?ABcZirTz?}                  & $(2,3,3,4,4,4,5,5,5,6,7)$ & $24$\\
$G_7$ & \texttt{J?\kern0pt\textasciigrave e\textasciigrave rcJ\textasciicircum h\_} & $(3,3,3,4,4,4,5,5,6,6,7)$ & $25$\\
$G_8$ & \texttt{J?ABFDyvfm\_}                 & $(2,3,3,4,4,5,5,5,6,7,8)$ & $26$\\
\bottomrule
\end{tabular}
\end{table}

We describe the first graph explicitly.  Let $G_1$ have vertex set $\{0,1,\dots,10\}$ and edge set
\begin{align*}
E(G_1)=\{&05,06,07,09,0\mathrm{X},15,17,1\mathrm{X},26,27,29,2\mathrm{X},36,38,3\mathrm{X},48,\\
&4\mathrm{X},58,59,69,79,7\mathrm{X},9\mathrm{X}\},
\end{align*}
where $\mathrm X$ denotes vertex $10$ and $uv$ abbreviates the edge $\{u,v\}$; so $|E(G_1)|=23$.  The graph is drawn in \cref{fig:g1}.

\begin{figure}[ht]
\centering
\begin{tikzpicture}[scale=0.95,
  vert/.style={circle,draw,fill=white,inner sep=1.2pt,font=\footnotesize,minimum size=14pt}]
\draw[semithick]
 (0.61,-2.17) -- (-1.04,-1.16)  
 (0.61,-2.17) -- (0.22,-0.74)   
 (0.61,-2.17) -- (2.38,-0.18)   
 (0.61,-2.17) -- (1.53,-1.21)   
 (0.61,-2.17) -- (0.31,0.49)    
 (0.61,1.80) -- (-1.04,-1.16)   
 (0.61,1.80) -- (2.38,-0.18)    
 (0.61,1.80) -- (0.31,0.49)     
 (2.21,1.03) -- (0.22,-0.74)    
 (2.21,1.03) -- (2.38,-0.18)    
 (2.21,1.03) -- (1.53,-1.21)    
 (2.21,1.03) -- (0.31,0.49)     
 (-1.74,0.29) -- (0.22,-0.74)   
 (-1.74,0.29) -- (-3.20,-0.11)  
 (-1.74,0.29) -- (0.31,0.49)    
 (-1.88,1.95) -- (-3.20,-0.11)  
 (-1.88,1.95) -- (0.31,0.49)    
 (-1.04,-1.16) -- (-3.20,-0.11) 
 (-1.04,-1.16) -- (1.53,-1.21)  
 (0.22,-0.74) -- (1.53,-1.21)   
 (2.38,-0.18) -- (1.53,-1.21)   
 (2.38,-0.18) -- (0.31,0.49)    
 (1.53,-1.21) -- (0.31,0.49);   
\node[vert] at (0.61,-2.17) {$0$};
\node[vert] at (0.61,1.80) {$1$};
\node[vert] at (2.21,1.03) {$2$};
\node[vert] at (-1.74,0.29) {$3$};
\node[vert] at (-1.88,1.95) {$4$};
\node[vert] at (-1.04,-1.16) {$5$};
\node[vert] at (0.22,-0.74) {$6$};
\node[vert] at (2.38,-0.18) {$7$};
\node[vert] at (-3.20,-0.11) {$8$};
\node[vert] at (1.53,-1.21) {$9$};
\node[vert] at (0.31,0.49) {$10$};
\end{tikzpicture}
\caption{The graph $G_1$ (graph6 code \texttt{J?BDf?[hvq\_}), the first of the eight counterexamples: $\mdim(G_1)=n(G_1)=11$.}\label{fig:g1}
\end{figure}
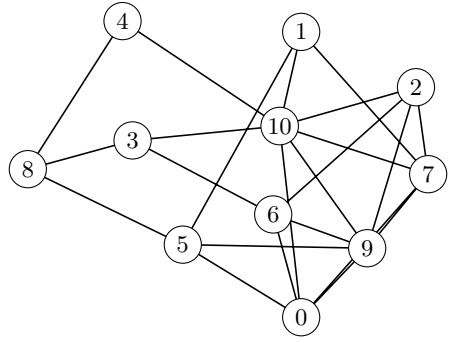

\begin{proposition}\label{prop:g1}
$V(G_1)$ is an $m$-resolving set of $G_1$, and no proper subset of $V(G_1)$ is $m$-resolving.  Consequently $\mdim(G_1)=11=n(G_1)$.
\end{proposition}

\begin{proof}[Proof of the first claim]
$G_1$ has diameter $3$, so $\mathrm{DDS}(v)=(1,a_1,a_2,a_3)$ with $a_1+a_2+a_3=10$, and by \cref{obs:ddi} it suffices to check that the triples $(a_1,a_2,a_3)$ are pairwise distinct.  They are:
\begin{center}
\small
\begin{tabular}{cc@{\qquad}cc@{\qquad}cc}
\toprule
$v$ & $(a_1,a_2,a_3)$ & $v$ & $(a_1,a_2,a_3)$ & $v$ & $(a_1,a_2,a_3)$\\
\midrule
$0$ & $(5,5,0)$ & $4$ & $(2,7,1)$ & $8$  & $(3,5,2)$\\
$1$ & $(3,6,1)$ & $5$ & $(4,6,0)$ & $9$  & $(6,4,0)$\\
$2$ & $(4,5,1)$ & $6$ & $(4,4,2)$ & $10$ & $(7,3,0)$\\
$3$ & $(3,7,0)$ & $7$ & $(5,4,1)$ &      &          \\
\bottomrule
\end{tabular}
\end{center}
All eleven triples are distinct, so $V(G_1)$ is $m$-resolving.  (The table is straightforward to verify from the edge list: for instance vertex $4$ has neighbours $\{8,10\}$, second neighbourhood $\{0,1,2,3,5,7,9\}$, and only vertex $6$ at distance $3$.)

The second claim --- that none of the $2^{11}-2=2046$ proper nonempty subsets of $V(G_1)$ is $m$-resolving --- is a finite check that we performed by computer with three independently written programs (\cref{sec:search}).  The same verification was applied to $G_2,\dots,G_8$.
\end{proof}

\begin{remark}
By \cref{obs:ddi} and the diameter-$3$ structure, all the information used by the full vertex set is the triple $(a_1,a_2,a_3)$; that a graph can be ``just barely'' identifiable --- through global distance counts only, with every smaller landmark multiset ambiguous --- is exactly the phenomenon the eight graphs exhibit.  All eight examples are asymmetric DDI graphs of diameter $3$ (trivial automorphism group is forced, \cref{sec:prelim}): their full distance degree sequences distinguish the vertices, while every proper landmark set leaves at least one unresolved pair.
\end{remark}

\section{The exhaustive search and its verification}\label{sec:search}

\subsection{Scope}
We determined, for every connected graph of order $2\le n\le 11$, whether $\mdim(G)\ge n-1$, and for the relevant graphs the exact value of $\mdim(G)$.  \Cref{tab:scan} summarizes the search.  Completeness of the enumeration rests on the generator: \texttt{geng} outputs one representative of every isomorphism class of connected graphs, and its \texttt{res/mod} streams are disjoint and jointly exhaustive.  The totals agree with the known numbers of connected graphs (OEIS \texttt{A001349} \cite{OEIS-A001349}) at every order, providing an end-to-end consistency check on the enumeration.

\begin{table}[ht]
\centering
\caption{Search summary.  ``max'' is the maximum finite value of $\mdim$ over all connected graphs of order $n$; the last two columns count the graphs attaining $n-1$ and $n$.}\label{tab:scan}
\small
\begin{tabular}{rrrrrr}
\toprule
$n$ & connected graphs & DDI graphs & max & $\mdim=n-1$ & $\mdim=n$\\
\midrule
$2$ & $1$ & $0$ & $1$ & $1$ & $0$\\
$3$ & $2$ & $0$ & $1$ & $0$ & $0$\\
$4$ & $6$ & $0$ & $1$ & $0$ & $0$\\
$5$ & $21$ & $0$ & $3$ & $0$ & $0$\\
$6$ & $112$ & $0$ & $4$ & $0$ & $0$\\
$7$ & $853$ & $10$ & $5$ & $0$ & $0$\\
$8$ & $11{,}117$ & $91$ & $6$ & $0$ & $0$\\
$9$ & $261{,}080$ & $1{,}474$ & $8$ & $29$ & $0$\\
$10$ & $11{,}716{,}571$ & $37{,}135$ & $9$ & $248$ & $0$\\
$11$ & $1{,}006{,}700{,}565$ & $1{,}437{,}134$ & $11$ & $3{,}776$ & $8$\\
\bottomrule
\end{tabular}
\end{table}

\subsection{Algorithm}
Graphs of order $n\le 7$ were taken from the atlas shipped with \textsf{networkx} and processed by direct enumeration of all vertex subsets.  Graphs of order $8\le n\le 10$ were read from B.~McKay's published \texttt{graph6} databases, and the $1{,}006{,}700{,}565$ connected graphs of order $11$ were generated with \texttt{geng} from \textsf{nauty}~2.8.9 \cite{McKayPiperno2014,NautySoftware}, split into eight parallel streams via the \texttt{res/mod} option with identical remaining parameters (\texttt{geng -cq 11 $i$/8} for $i=0,\dots,7$).

The scanner (C++, ancillary file \texttt{scan11.cpp}) processes each graph as follows.
\begin{enumerate}[leftmargin=2em]
\item Compute all-pairs distances by BFS over bitmask adjacency lists.
\item Encode $r_{\mathrm m}(v\mid W)$ as the $64$-bit integer $\sum_{w\in W} 16^{\,\dist(v,w)}$.  This packs the multiplicity of each distance value into one hexadecimal digit; it is collision-free here because multiplicities are at most $n\le 11<16$ and $4(\operatorname{diam}+1)\le 44<64$ bits suffice.  A set $W$ is $m$-resolving iff the $n$ encodings are pairwise distinct.
\item Filter: if $G$ is not DDI \emph{and} no $(n-1)$-subset of $V(G)$ is $m$-resolving, skip the graph.  This is sound for our purpose: $\mdim(G)=n$ requires $V(G)$ to resolve (\cref{cor:char}), and $\mdim(G)=n-1$ requires some $(n-1)$-subset to resolve; skipped graphs therefore satisfy $\mdim(G)\le n-2$ or $\mdim(G)=\infty$.
\item For the surviving graphs, compute $\mdim(G)$ exactly by enumerating subsets in increasing size and reporting the first resolving one.
\end{enumerate}
The full order-$11$ scan took about $19$ minutes of wall time using eight parallel processes on an AMD Ryzen~5~5600H system ($6$ cores/$12$ threads).

\subsection{Independent verification}
The eight candidates reported with $\mdim(G)=11$ were re-verified by a second, independently written program (\textsf{Python}/\textsf{networkx}, ancillary file \texttt{verify\_counterexamples.py}) sharing no code with the scanner: it re-parses the \texttt{graph6} strings, recomputes all distances, checks that $V(G)$ resolves, and exhaustively confirms that none of the $2046$ proper nonempty subsets resolves.  All eight graphs passed.  A third implementation, written for the final audit (ancillary file \texttt{audit\_third\_check.py}), avoids BFS altogether: it parses \texttt{graph6} from the format specification and recomputes all distances via boolean matrix powers; it reconfirms all eight counterexamples, the triple table of \cref{prop:g1}, and the order-$8$ example of \cref{sec:prelim}.  Pairwise non-isomorphism was confirmed directly with \textsf{networkx} for all $28$ pairs; it is also consistent with \texttt{geng} emitting one representative per isomorphism class and the eight graphs arising in its disjoint enumeration streams.

The extremal counts of \cref{tab:scan} are certified by the raw scanner outputs: the graphs reported with $\mdim(G)\in\{n-1,n\}$ at orders $10$ and $11$ are listed, one per line, in the ancillary output files, and the aggregation script \texttt{aggregate\_extremal\_counts.py} recounts them ($248$ at order $10$; $3{,}776$ and $8$ at order $11$), checks the eight $\mdim=11$ codes against \cref{tab:eight}, and re-sums the per-stream totals and the DDI distribution.

Beyond the counterexamples themselves, two soundness checks are worth recording.  First, the multiset convention: our encodings include the entry $\dist(v,v)=0$ when $v\in W$, matching the definition of \cite{SimanjuntakEtAl} (this convention is essential; it is what forces $\mdim(G)\neq 2$).  Second, the DDI distribution: at $n=11$ the $1{,}437{,}134$ DDI graphs have multiset dimensions distributed as
\[
3^{\,244621}\;4^{\,996380}\;5^{\,185858}\;6^{\,9514}\;7^{\,630}\;8^{\,93}\;9^{\,29}\;10^{\,1}\;11^{\,8}
\]
(exponents denote counts).  At order $11$ the overwhelming majority of DDI graphs thus have \emph{small} multiset dimension; only eight attain the opposite extreme, a fraction $8/1{,}006{,}700{,}565\approx 8\cdot 10^{-9}$ of all connected graphs of order $11$.

\subsection{Proof of the main theorem}
\begin{proof}[Proof of \cref{thm:order}]
The sources of \cref{tab:scan} enumerate every connected graph of orders $2$ through $11$ exactly once up to isomorphism: the \textsf{networkx} atlas for $n\le7$, McKay's databases for $8\le n\le10$, and the disjoint, jointly exhaustive \texttt{geng} streams for $n=11$, with all totals matching the known connected-graph counts.  For each graph the scanner computes all pairwise distances by BFS; the encoding of step~(2) represents $r_{\mathrm m}(v\mid W)$ injectively, because every multiplicity is at most $n\le11<16$ and so occupies a single hexadecimal digit.  The filter of step~(3) discards a graph only if $V(G)$ does not resolve \emph{and} no $(n-1)$-subset resolves, and by \cref{cor:char} such a graph satisfies neither $\mdim(G)=n$ nor $\mdim(G)=n-1$.  For every surviving graph the exact value of $\mdim$ is obtained by complete subset enumeration.  The scan reports exactly eight graphs with $\mdim(G)=n(G)$, all of order $11$ (\cref{tab:eight}), and none of smaller order; the aggregation script recounts these from the raw outputs.  The eight graphs, their invariants, and pairwise non-isomorphism were re-verified by the two independent programs described above.
\end{proof}

\section{Concluding remarks on Part I}\label{sec:remarks}

\subsection{The corrected picture}
The correct universal upper bound for graphs with finite multiset dimension is the trivial one, $\mdim(G)\le n(G)$, and it is attained --- first at order $11$.  \Cref{tab:scan} also corrects the intuition at the second-highest value: $\mdim(G)=n-1$ is attained at $n=2$ (by $K_2$), never for $3\le n\le 8$, and then by $29$, $248$, and $3{,}776$ graphs at $n=9,10,11$ respectively.  The extremal families grow rapidly once they appear.
\subsection{Towards larger orders}
We checked all $8(2^{11}-1)=16{,}376$ nonempty neighbourhood extensions of the eight labelled seed graphs by one new vertex (allowing repetition up to isomorphism): none of the resulting DDI graphs of order $12$ has $\mdim=12$ (the best value attained is $11=n-1$, by an extension of $G_2$).  A full scan of the $\sim 1.6\times 10^{11}$ connected graphs of order $12$ is beyond the scope of this paper.  We leave the natural questions open.

\begin{problem}\label{prob:all-n}
Does a graph of order $n$ with $\mdim(G)=n$ exist for infinitely many $n$?  More strongly, does one exist for every $n\ge 11$?
\end{problem}

\begin{problem}\label{prob:struct}
Characterize the graphs with $\mdim(G)=n(G)$, or find an infinite family.  All eight known examples have diameter $3$ and girth $3$; is diameter $3$ necessary apart from finitely many exceptions?
\end{problem}

\begin{problem}\label{prob:count}
The number of graphs of order $n$ with $\mdim=n-1$ appears to grow quickly ($29$, $248$, $3{,}776$ for $n=9,10,11$).  Determine its asymptotics, and likewise for $\mdim=n$.
\end{problem}

\partheading{Part II. Square king grids}

Only the lower bound of \cref{thm:square} is new.  The proof shows that every set of three landmarks creates a collision.  The argument is stronger than a counting proof: for each geometric normal form it gives an explicit pair of vertices whose ordered distance vectors are either equal or differ only by a permutation.

\section{King grids and rotated coordinates}\label{sec:rotated}

Fix $n\ge 5$ and write $m=n-1$.  We identify the vertex set of the king grid with
\[
 \King_m=\{0,1,\dots,m\}^2.
\]
The graph distance is the Chebyshev metric
\begin{equation}\label{eq:cheb}
 \dist_\infty((x,y),(a,b))=\max\{|x-a|,|y-b|\}.
\end{equation}

Introduce rotated coordinates
\begin{equation}\label{eq:rotation}
 u=x+y-m,
 \qquad
 v=x-y.
\end{equation}
The image of $\King_m$ is the parity sublattice
\begin{equation}\label{eq:diamond}
 \D_m=\{(u,v)\in\mathbb Z^2: |u|+|v|\le m,\ u+v\equiv m\pmod 2\}.
\end{equation}
The inverse transformation is
\begin{equation}\label{eq:inverse}
 x=\frac{u+v+m}{2},
 \qquad
 y=\frac{u-v+m}{2}.
\end{equation}

\begin{lemma}\label{lem:l1}
If $z,z'\in\King_m$ have rotated coordinates $(u,v)$ and $(u',v')$, then
\[
 2\dist_\infty(z,z')=|u-u'|+|v-v'|.
\]
\end{lemma}

\begin{proof}
Put $\alpha=x-x'$ and $\beta=y-y'$.  Then $u-u'=\alpha+\beta$ and $v-v'=\alpha-\beta$.  The elementary identity
\[
 |\alpha+\beta|+|\alpha-\beta|=2\max\{|\alpha|,|\beta|\}
\]
together with \eqref{eq:cheb} gives the claim.
\end{proof}

We shall freely use the symmetries of the square, which in $(u,v)$-coordinates are generated by sign changes and the interchange of $u$ and $v$.

\section{A boundary obstruction}\label{sec:boundary}

Let $S=\{p_1,p_2,p_3\}\subseteq\D_m$, where $p_i=(u_i,v_i)$.  Set
\[
 u_- =\min_i u_i,
 \quad u_+=\max_i u_i,
 \quad v_- =\min_i v_i,
 \quad v_+=\max_i v_i.
\]
Every multiset resolving set is, in particular, an ordinary resolving set.  Looking only at adjacent pairs on the four sides of the square gives the following necessary conditions.

\begin{lemma}[Boundary inequalities]\label{lem:boundary}
If $S$ is an ordinary resolving set of $\King_m$, then
\begin{equation}\label{eq:boundaryineq}
\begin{aligned}
 u_-+v_-&\le -m, &
 u_- -v_+&\le -m,\\
 u_++v_+&\ge m, &
 u_+-v_-&\ge m.
\end{aligned}
\end{equation}
\end{lemma}

\begin{proof}
We prove the last inequality; the other three follow by rotations and reflections.  Consider adjacent vertices $(x,m)$ and $(x+1,m)$ on the top side, where $0\le x<m$.  Let a landmark $(a,b)$ have rotated coordinates
\[
 u_0=a+b-m,
 \qquad
 v_0=a-b.
\]
A direct calculation from \eqref{eq:cheb} shows that
\[
 \dist_\infty((x,m),(a,b))
 =\dist_\infty((x+1,m),(a,b))
\]
if and only if
\[
 u_0\le x\le m+v_0-1.
\]
Hence, if $u_+-v_-\le m-1$, then the integer interval
\[
 [u_+,m+v_--1]
\]
is nonempty.  Moreover, its left endpoint is at most $m-1$: if $u_+=m$, then a landmark attaining $u_+$ must have $v=0$, whence $v_-\le0$ and $u_+-v_-\ge m$.  Its right endpoint is at least $0$: if $v_-=-m$, then a landmark attaining $v_-$ must have $u=0$, whence $u_+\ge0$ and again $u_+-v_-\ge m$.  Both alternatives contradict $u_+-v_-\le m-1$.  We may therefore choose
\[
 x\in [u_+,m+v_--1]\cap\{0,\dots,m-1\}.
\]
This $x$ is simultaneously contained in the three landmark intervals, and the adjacent top-side vertices $(x,m)$ and $(x+1,m)$ have equal ordered distance vectors to $S$.  This contradicts ordinary resolvability.  Therefore $u_+-v_-\ge m$.
\end{proof}

\section{Antipodal landmarks}\label{sec:antipodal}

\begin{lemma}[Antipodal obstruction]\label{lem:antipodal}
Suppose two landmarks in $\D_m$ are antipodal, say
\[
 A=(p,q),
 \qquad
 C=(-p,-q).
\]
Then no choice of a third landmark makes $\{A,B,C\}$ a multiset resolving set.
\end{lemma}

\begin{proof}
Write $B=(r,s)$.  If $B\ne(0,0)$, define
\[
 X=(s,-r),
 \qquad
 Y=(-s,r).
\]
Since $|r|+|s|\le m$ and the parity condition is preserved, $X,Y\in\D_m$.  Moreover, $X\ne Y$.  The half-turn interchanges $A$ and $C$, hence
\[
 \dist(X,A)=\dist(Y,C),
 \qquad
 \dist(X,C)=\dist(Y,A).
\]
Also, by \cref{lem:l1},
\[
 2\dist(X,B)=|s-r|+|r+s|=2\dist(Y,B).
\]
Thus the three distances from $X$ and $Y$ agree as multisets.

If $B=(0,0)$, then $A$ and $C$ themselves already have the same multiset representation with respect to $\{A,B,C\}$, because they are equidistant from $B$.
\end{proof}

\section{The case of distinct rotated coordinates}\label{sec:distinct}

Assume throughout this section that all $u_i$ are distinct and all $v_i$ are distinct.  Relabel the landmarks so that
\[
 u_1<u_2<u_3.
\]
The relative order of $(v_1,v_2,v_3)$ is a permutation of $123$.  Under the symmetries $(u,v)\mapsto(-u,v)$, $(u,v)\mapsto(u,-v)$, and $(u,v)\mapsto(v,u)$, the six permutations split into two orbits: the monotone permutations $123,321$, and the four zigzag permutations $132,213,231,312$.  It is therefore enough to treat one representative of each orbit.

\begin{lemma}\label{lem:distinct}
If all three $u$-coordinates and all three $v$-coordinates are distinct, then the three landmarks do not form a multiset resolving set.
\end{lemma}

\begin{proof}
Suppose, for a contradiction, that the three landmarks form a multiset resolving set.  In particular they form an ordinary resolving set, so the boundary inequalities of \cref{lem:boundary} hold.

Suppose first that
\[
 v_1<v_2<v_3.
\]
By \cref{lem:boundary},
\[
 u_1+v_1\le -m,
 \qquad
 u_3+v_3\ge m.
\]
On the other hand, membership in $\D_m$ gives the reverse inequalities.  Hence
\[
 u_1+v_1=-m,
 \qquad
 u_3+v_3=m.
\]
The remaining two boundary inequalities become
\[
 u_1-v_3\le -m,
 \qquad
 u_3-v_1\ge m.
\]
Substituting the two equalities above gives $v_1+v_3\ge0$ and $v_1+v_3\le0$.  Therefore
\[
 (u_3,v_3)=-(u_1,v_1),
\]
and \cref{lem:antipodal} applies.  The decreasing monotone order is equivalent by reflection.

It remains to consider the zigzag orbit.  By symmetry, assume
\begin{equation}\label{eq:zigzagorder}
 v_3<v_1<v_2.
\end{equation}
The last inequality in \eqref{eq:boundaryineq}, together with $p_3\in\D_m$, forces
\begin{equation}\label{eq:p3edge}
 u_3-v_3=m.
\end{equation}
Indeed, $u_3-v_3\ge m$, while $u_3-v_3\le |u_3|+|v_3|\le m$.  In particular $u_3\ge0\ge v_3$ and $v_3=u_3-m$.  Moreover $u_3\ge1$: if $u_3=0$, then $u_+=0$, and the third boundary inequality gives $v_+=v_2=m$, which forces $u_2=0=u_3$, contradicting the strict ordering of the $u_i$.

If $u_2\le u_3-2$, then the two vertices
\[
 P=(u_3-2,v_3),
 \qquad
 Q=(u_3-1,v_3+1)
\]
belong to $\D_m$ and have equal ordered distance vectors to all three landmarks.  Membership uses $u_3\ge1$: since $v_3=u_3-m$, we get $|u_3-2|+|v_3|=m-2$ for $u_3\ge2$ and $=m$ for $u_3=1$, while $|u_3-1|+|v_3+1|=m-2$ for $u_3\le m-1$ and $=m$ for $u_3=m$; both points inherit the parity of $p_3$.  For $i=1,2$, the $u$-distance increases by one from $P$ to $Q$ while the $v$-distance decreases by one; for $i=3$, both $L^1$-distances equal $2$.  Hence ordinary resolvability requires
\begin{equation}\label{eq:uadj}
 u_2=u_3-1.
\end{equation}
We next show that ordinary resolvability also requires
\begin{equation}\label{eq:vadj}
 v_1=v_3+1.
\end{equation}
Indeed, suppose that $v_1\ge v_3+2$.  In view of \eqref{eq:uadj}, the two vertices
\[
 P'=(u_3,v_3+2),
 \qquad
 Q'=(u_3-1,v_3+1)
\]
belong to $\D_m$.  To see this, note first that $v_3\ne0$: otherwise $u_3=m$, and the strict inequalities $0=v_3<v_1<v_2$ are incompatible with $p_2=(m-1,v_2)\in\D_m$.  Hence $v_3\le-1$.  If $v_3=-1$, then $P'=(m-1,1)$ lies on the boundary of the diamond; if $v_3\le-2$, then $|P'_u|+|P'_v|=m-2$.  In either case $P'\in\D_m$, while $|Q'_u|+|Q'_v|=m-2$; parity is preserved for both points.  For $i=1,2$, the $u$-distance decreases by one from $P'$ to $Q'$ while the $v$-distance increases by one, and both points have $L^1$-distance $2$ from $p_3$.  Thus $P'$ and $Q'$ have equal ordered distance vectors, a contradiction.  Since $v_1>v_3$, \eqref{eq:vadj} follows.

Write
\[
 v_3=-a-1,
 \qquad
 L=m-a.
\]
Then \eqref{eq:p3edge} gives $u_3=L-1$.  Relations \eqref{eq:uadj} and \eqref{eq:vadj}, the boundary inequalities, the diamond constraint, and the parity condition yield
\begin{equation}\label{eq:zigzagnormal}
\begin{aligned}
 p_1&=(-L,-a),\\
 p_2&=(L-2,a+2),\\
 p_3&=(L-1,-a-1),
\end{aligned}
\qquad
0\le a\le m-2.
\end{equation}
For completeness, the only numerical choices left by the inequalities are
\[
 -L\le u_1\le -L+1,
 \qquad
 a+1\le v_2\le a+2;
\]
the parity condition selects $u_1=-L$ and $v_2=a+2$.

Returning to the original coordinates via \eqref{eq:inverse}, the landmarks are
\[
 (0,a),
 \qquad
 (m,m-a-2),
 \qquad
 (m-a-1,m).
\]
The distinct vertices
\[
 U=(m-a-2,m),
 \qquad
 V=(m-a,m)
\]
have the same ordered distance vector
\[
 (L,a+2,1).
\]
Thus the zigzag case is not even ordinarily resolving.
\end{proof}

\section{The case of a repeated rotated coordinate}\label{sec:repeated}

By symmetry, it remains to consider landmarks for which two $v$-coordinates are equal.  Write
\[
 p_1=(r_1,t),
 \qquad
 p_2=(r_2,t),
 \qquad
 p_3=(r_3,s),
\]
where $r_1<r_2$ and, after reflection, $s>t$.

If all three $v$-coordinates are equal, \cref{lem:boundary} and the diamond constraints force $t=0$, $u_-=-m$, and $u_+=m$, so the landmark set contains an antipodal pair.  We may therefore assume that exactly two $v$-coordinates are equal.

\begin{lemma}[Normal form]\label{lem:normalform}
If the three landmarks satisfy the boundary inequalities and contain exactly two equal $v$-coordinates, then either they contain an antipodal pair or, after a symmetry of the square, they have the following form in the original coordinates:
\begin{equation}\label{eq:normalform}
 A=(0,a),
 \qquad
 C=(L,m),
 \qquad
 B=(b,c),
\end{equation}
where
\begin{equation}\label{eq:normalconstraints}
 L=m-a>0,
 \qquad
 0\le b,c\le m,
 \qquad
 b-c\ge a.
\end{equation}
\end{lemma}

\begin{proof}
If $r_3>r_2$, then $u_-=r_1$, $u_+=r_3$, $v_-=t$, and $v_+=s$.  The first and third boundary inequalities are equalities because the corresponding points lie in $\D_m$:
\[
 r_1+t=-m,
 \qquad
 r_3+s=m.
\]
The other two inequalities imply $s+t\ge0$ and $s+t\le0$.  Hence $s=-t$ and $r_3=-r_1$, so $p_1$ and $p_3$ are antipodal.  The case $r_3<r_1$ is symmetric.

We may therefore assume $r_1\le r_3\le r_2$.  Now $u_-=r_1$, $u_+=r_2$, $v_-=t$, and $v_+=s$.  The first and fourth boundary inequalities, together with the diamond constraints for $p_1$ and $p_2$, force
\[
 r_1+t=-m,
 \qquad
 r_2-t=m.
\]
Thus, for some $a\in\{0,\dots,m-1\}$ and $L=m-a$,
\[
 p_1=(-L,-a),
 \qquad
 p_2=(L,-a).
\]
The remaining boundary inequalities imply $s\ge a$.  Under the inverse map \eqref{eq:inverse}, the two equal-$v$ landmarks become $A=(0,a)$ and $C=(L,m)$, while $p_3=(r_3,s)$ becomes $B=(b,c)$ with $s=b-c\ge a$.  This is exactly \eqref{eq:normalform}--\eqref{eq:normalconstraints}.
\end{proof}

We now eliminate the normal form by explicit collisions.

\begin{lemma}[Collision table]\label{lem:collisiontable}
No landmark set satisfying \eqref{eq:normalform}--\eqref{eq:normalconstraints} is multiset resolving.
\end{lemma}

\begin{proof}
First suppose that $B$ lies in the axis-parallel square with opposite corners $A$ and $C$, that is,
\[
 b\le L,
 \qquad
 c\ge a.
\]
Take
\[
 U=(b,a),
 \qquad
 V=(0,c).
\]
Then
\[
 r(U\mid A,B,C)=(b,c-a,L),
 \qquad
 r(V\mid A,B,C)=(c-a,b,L),
\]
so the multisets agree.  The vertices $U$ and $V$ are distinct: equality would force $b=0$ and $a=c$, and then the constraint $b-c\ge a$ would give $a=c=0$, contrary to the landmarks $A$ and $B$ being distinct.

Assume now that $B$ lies outside this square.  The reflection
\[
 \rho(x,y)=(m-y,m-x)
\]
interchanges $A$ and $C$.  After applying $\rho$ if necessary, we may assume
\begin{equation}\label{eq:outside-normal}
 c<a,
 \qquad
 b+c\le m.
\end{equation}
Indeed, if $c\ge a$ then $b>L$ and hence $b+c>m$, while if $c<a$ but $b+c>m$, the same reflection produces \eqref{eq:outside-normal}.

Set
\[
 \delta=a-c>0,
 \qquad
 x=b-L.
\]
The constraints imply $x\le\delta$.  If $\delta>L$, they also imply $x\ge0$.  The five rows of the following table are pairwise disjoint and exhaustive: when $\delta\le L$ the trichotomy $b<L$, $b=L$, $b>L$ applies, and when $\delta>L$ (so $x\ge0$) either $x\le2L$ or $x>2L$.  Each row gives two distinct grid vertices with the same distance multiset to $A,B,C$.

\medskip
\begin{center}
\renewcommand{\arraystretch}{1.35}
\begin{tabularx}{\textwidth}{>{\raggedright\arraybackslash}p{0.23\textwidth} >{\centering\arraybackslash}p{0.18\textwidth} >{\centering\arraybackslash}p{0.18\textwidth} X}
\toprule
Condition & $U$ & $V$ & Common multiset \\
\midrule
$\delta\le L$, $b<L$
& $(b,m)$
& $(L,c)$
& $\ms{L,L+\delta,L-b}$ \\
$\delta\le L$, $b=L$
& $B$
& $C$
& $\ms{0,L,L+\delta}$ \\
$\delta\le L$, $b>L$
& $(L-\delta,m)$
& $(x,c)$
& $\ms{\delta,L,L+\delta}$ \\
$\delta>L$, $x\le2L$
& $(x,c+L)$
& $(x,a)$
& $\ms{x,L,\delta}$ \\
$\delta>L$, $x>2L$
& $(2L+x-\delta,c)$
& $(\delta,m)$
& $\ms{\delta,\delta-L,L+\delta}$ \\
\bottomrule
\end{tabularx}
\end{center}
\medskip

We verify the rows briefly.  In the first row, for example,
\[
 r(U\mid A,B,C)=(L,L+\delta,L-b),
\]
whereas
\[
 r(V\mid A,B,C)=(L,L-b,L+\delta).
\]
The second row is immediate.  In the third row, the vectors are
\[
 (L,L+\delta,\delta)
 \quad\text{and}\quad
 (\delta,L,L+\delta).
\]
In the fourth row they are
\[
 (x,L,\delta)
 \quad\text{and}\quad
 (x,\delta,L),
\]
and in the fifth row they are
\[
 (\delta,\delta-L,L+\delta)
 \quad\text{and}\quad
 (\delta,L+\delta,\delta-L).
\]
The inequalities in each row ensure that all listed coordinates lie in $\{0,\dots,m\}$ and that the displayed maxima have the stated values.  In every row the two vertices are distinct: in the first row because $b<L$, in the second because $B\ne C$, in the third and fifth because the second coordinates are $m$ and $c<a\le m$, and in the fourth because $c+L=a$ would mean $\delta=L$.  Hence every normal-form triple has a collision.
\end{proof}

\section{Proof of \texorpdfstring{\cref{thm:square}}{the square king grid theorem}}\label{sec:proofsquare}

\begin{proof}[Proof of \cref{thm:square}]
Hakanen and Yero~\cite[Proposition~4.2 and Theorem~4.3]{HakanenYero2024} proved the upper bound
\[
 \mdim(P_n\boxtimes P_n)\le4
 \qquad (n\ge5).
\]
It remains to rule out a multiset resolving set of cardinality three.

Let $S$ be any three vertices.  If $S$ is not an ordinary resolving set, then it is certainly not multiset resolving.  Otherwise the boundary inequalities of \cref{lem:boundary} hold.  If all three $u$-coordinates and all three $v$-coordinates are distinct, \cref{lem:distinct} supplies a collision.  If a rotated coordinate is repeated, a symmetry of the square reduces to the case of repeated $v$-coordinates.  If all three $v$-coordinates are equal, the observation opening \cref{sec:repeated} produces an antipodal pair and \cref{lem:antipodal} applies.  If exactly two are equal, \cref{lem:normalform} yields either an antipodal pair, handled again by \cref{lem:antipodal}, or the normal form, in which case the collision is supplied by \cref{lem:collisiontable}.  Thus no three-vertex multiset resolving set exists.

Since $P_n\boxtimes P_n$ is not a path, its multiset dimension cannot be $1$, and no graph has multiset dimension $2$~\cite{SimanjuntakEtAl}.  Therefore the upper bound $4$ is exact.
\end{proof}

Combining the theorem with the small cases already established in~\cite{HakanenYero2024} gives the complete classification.

\begin{corollary}\label{cor:classification}
For every integer $n\ge2$,
\[
 \mdim(P_n\boxtimes P_n)=
 \begin{cases}
 \infty, & n=2,3,\\
 6, & n=4,\\
 4, & n\ge5.
 \end{cases}
\]
\end{corollary}

\section{Beyond squares: rectangular king grids}\label{sec:rect}

Problem~1 of the survey~\cite{FarhanEtAl2026} asks, more generally, for $\mdim(P_{n_1}\boxtimes P_{n_2})$ for all $n_1,n_2\ge3$.  The method of this part does not settle the rectangular case, and we record here why.  \Cref{lem:l1,lem:boundary} generalise to rectangles with routine changes.  The antipodal obstruction of \cref{lem:antipodal}, however, constructs the colliding pair $X=(s,-r)$, $Y=(-s,r)$ by a quarter-turn of the third landmark, and a quarter-turn is a symmetry of the square only.  This is not a cosmetic issue: the entire monotone case and part of the repeated-coordinate case funnel into \cref{lem:antipodal}.

Computational evidence, reported in the rectangular landscape of \cref{sec:landscape} (Part~III), shows that the rectangular answer is genuinely different.  Three phenomena stand out.  The value $3$ never occurs in the surveyed range, so the lower bound of \cref{thm:square} plausibly extends to all rectangles with $\min(n_1,n_2)\ge5$.  The value $4$ occurs only in a narrow band around the diagonal, while farther from the diagonal the observed values grow.  Finally, the parameter is not monotone along a row: $\mdim(P_4\boxtimes P_{10})>5$ while $\mdim(P_4\boxtimes P_{11})=5$, both cells confirmed by full enumeration in two independent implementations; the height-$4$ scan of \cref{sec:landscape} sharpens the former value to $\mdim(P_4\boxtimes P_{10})=6$.  Part~III determines the height-three case exactly and shows linear growth for every fixed height; a complete determination of the rectangular king grids appears to require new constructions as well as new obstructions, and we leave it as an open problem.

\section{Computational verification of Part II}\label{sec:verifsquare}

The proof above is independent of computation.  Nevertheless, verifiers accompany this manuscript; together with their outputs they are permanently archived (see the data availability statement).  They perform three checks:
\begin{enumerate}[label=(\roman*)]
\item exhaustive enumeration of all three-landmark sets for small square king grids, checking that none is multiset resolving;
\item direct verification of every collision identity in \cref{lem:collisiontable} over a user-selected finite parameter range;
\item the rectangular scan underlying the discussion of \cref{sec:rect}, with independent spot-checks of the extremal cells.
\end{enumerate}
These checks are useful for detecting parity mistakes, omitted boundary cases, or incorrect maxima, but they are not used to infer the infinite family theorem.

\begin{remark}
The collision proof often establishes more than needed.  In the zigzag case the two vertices have identical \emph{ordered} distance vectors.  In the normal-form table, the two ordered vectors are explicit permutations of one another, which is precisely the obstruction created by forgetting landmark labels.
\end{remark}

\partheading{Part III. King strips}

A king strip of height three requires one landmark per column (\cref{thm:strip}), whereas an $n\times n$ king square is resolved by four landmarks (\cref{thm:square}).  The reason is geometric: in a narrow strip, all landmarks are vertically close to every vertex, so distant landmarks measure only horizontal displacement and are blind to the vertical coordinate.  This forces a dense local structure that we capture in three exact per-column conditions (\cref{sec:local}), and the conditions alone already force $n$ landmarks (\cref{sec:lower}).  Conversely, an explicit periodic pattern with exactly one landmark per column on average is a multiset resolving set (\cref{sec:upper}).

In addition, we determine all equality cases of the local lower-bound system (\cref{sec:extremal}): the column sequences attaining the bound are exactly the paths of an explicit finite automaton with a $19$-state recurrent core.  This yields necessary structural restrictions on minimum multiset resolving sets: at most two landmarks per column, two only as the corner pair of a column, with middle-row landmarks always alone in their column.  The local conditions are necessary but not sufficient for resolving; the automaton therefore describes a candidate family containing every minimum multiset resolving set.

The same blindness mechanism, combined with the fact that our landmark pattern applies to every height, yields the two-sided linear bound of \cref{thm:strips} for every fixed height.  \Cref{sec:landscape} reports computed values on rectangles of height up to $7$ (exact whenever the dimension is at most $6$, certified lower bounds ${>}6$ otherwise), which display an irregular transition zone; for instance
\[
 \mdim(P_4\boxtimes P_{11})=5,\quad
 \mdim(P_4\boxtimes P_{12})=6,\quad
 \mdim(P_4\boxtimes P_{13})>6,\quad
 \mdim(P_4\boxtimes P_{14})=6 .
\]

\subsection*{Notation}
Vertices of $P_h\boxtimes P_n$ are written $(x,y)$ with column $x\in\{0,\dots,n-1\}$ and row $y\in\{0,\dots,h-1\}$ (so $y\in\{0,1,2\}$ for the height-three strip); the distance is the Chebyshev metric
\[
 \dist\big((x,y),(a,b)\big)=\max\big(|x-a|,\,|y-b|\big).
\]
For a landmark set $S$ we write $p_x,m_x,q_x\in\{0,1\}$ for the indicators of $(x,0)\in S$, $(x,1)\in S$, $(x,2)\in S$, and we use the convention $p_x=m_x=q_x=0$ for $x\notin\{0,\dots,n-1\}$.  All multisets are denoted $\ms{\cdot}$.

\section{Three local separation conditions}\label{sec:local}

Throughout this section $S$ is a multiset resolving set of $P_3\boxtimes P_n$.  The reflection $\sigma(x,y)=(x,2-y)$ is an automorphism of $P_3\boxtimes P_n$.  The following two lemmas, giving three per-column conditions, isolate exactly which landmarks can separate the three vertices of one column, and what they must satisfy.  In each case the point is that a landmark contributes the \emph{same} value to both vertices of the pair unless it lies in a narrow window, and that the contributions inside the window can compensate each other only in the stated exceptional configuration.

\begin{lemma}[outer pair]\label{lem:A}
Let $x$ be a column.  If
\[
 \big(p_x,\;p_{x-1}+p_{x+1}\big)=\big(q_x,\;q_{x-1}+q_{x+1}\big),
\]
then $m(x,0\mid S)=m(x,2\mid S)$.  Consequently every column of a multiset resolving set satisfies
\[
 \textup{(A)}\qquad \big(p_x,\;p_{x-1}+p_{x+1}\big)\ne\big(q_x,\;q_{x-1}+q_{x+1}\big).
\]
\end{lemma}

\begin{proof}
Fix a landmark $s=(a,b)$ and write $h=|x-a|$.  The values of $s$ at the pair are $\max(h,b)$ and $\max(h,2-b)$.  If $b=1$ these coincide.  If $h\ge 2$ they equal $h$ in both cases, since $b,2-b\le 2\le h$.  Hence a landmark contributes distinct values only if $b\in\{0,2\}$ and $h\le 1$, in which case the values are $\ms{h,2}$, with $h$ going to the vertex on the landmark's side.

Write $u_0=p_x$, $u_1=p_{x-1}+p_{x+1}$, $w_0=q_x$, $w_1=q_{x-1}+q_{x+1}$.  Collecting contributions, the multiset $m(x,0\mid S)$ contains the sub-multiset
$\ms{0^{u_0},1^{u_1},2^{w_0+w_1}}$ from the window landmarks, and $m(x,2\mid S)$ contains $\ms{0^{w_0},1^{w_1},2^{u_0+u_1}}$, while all other landmarks contribute identical values to both.  If $(u_0,u_1)=(w_0,w_1)$, the two sub-multisets coincide and $m(x,0\mid S)=m(x,2\mid S)$.
\end{proof}

\begin{lemma}[inner pairs]\label{lem:BC}
Let $x$ be a column.  Every multiset resolving set satisfies
\begin{align*}
 \textup{(B)}&\qquad p_x\ne m_x \quad\text{or}\quad q_{x-1}+q_x+q_{x+1}\ge 1,\\
 \textup{(C)}&\qquad q_x\ne m_x \quad\text{or}\quad p_{x-1}+p_x+p_{x+1}\ge 1 .
\end{align*}
\end{lemma}

\begin{proof}
We prove (B); (C) follows by the reflection $\sigma$.  Consider the pair $(x,0)$, $(x,1)$ and a landmark $s=(a,b)$ with $h=|x-a|$.  The values are $\max(h,b)$ and $\max(h,|1-b|)$.  For $b=0$ they are $h$ and $\max(h,1)$, distinct only if $h=0$, with values $(0,1)$.  For $b=1$ they are $\max(h,1)$ and $h$, distinct only if $h=0$, with values $(1,0)$.  For $b=2$ they are $\max(h,2)$ and $\max(h,1)$, distinct only if $h\le 1$, with values $(2,1)$.

Hence, with $Q=q_{x-1}+q_x+q_{x+1}$, the differing contributions received by $(x,0)$ form the multiset $\ms{0^{p_x},1^{m_x},2^{Q}}$ and those received by $(x,1)$ form $\ms{1^{p_x},0^{m_x},1^{Q}}$.  These coincide if and only if $p_x=m_x$ and $Q=0$: comparing multiplicities of the value $0$ gives $p_x=m_x$; comparing the value $2$ gives $Q=0$; and conversely.  Thus if (B) fails, then $m(x,0\mid S)=m(x,1\mid S)$.
\end{proof}

\section{The lower bound}\label{sec:lower}

Conditions (A), (B), (C) concern only the indicator triples of consecutive columns.  For the transfer argument we reindex the columns by $1,\dots,n$: set $t_i=(p_{i-1},m_{i-1},q_{i-1})\in\{0,1\}^3$ for $1\le i\le n$, and put $t_0=t_{n+1}=(0,0,0)$.  Define the weight of the sequence $t_1,\dots,t_n$ as $\sum_{i=1}^n|t_i|$, where $|t|$ is the number of ones in $t$, and call the sequence \emph{admissible} if (A), (B), (C) hold at every position $i\in\{1,\dots,n\}$, read on the triples $(t_{i-1},t_i,t_{i+1})$.  By \cref{lem:A,lem:BC}, the indicator triples of every multiset resolving set of $P_3\boxtimes P_n$ form an admissible sequence, whose weight is exactly the number of landmarks.  The lower bound of \cref{thm:strip} therefore follows from:

\begin{proposition}\label{prop:lb}
Every admissible sequence of length $n\ge 3$ has weight at least $n$.
\end{proposition}

\begin{proof}
For a pair $(u,v)$ of triples let $\mathcal T$ be the transfer operator on vectors $V:\{0,1\}^3\times\{0,1\}^3\to\mathbb Z\cup\{\infty\}$ given by
\begin{multline*}
 (\mathcal T V)(v,w)\;=\;\min\big\{\,V(u,v)+|w| \;:\; u\in\{0,1\}^3,\\
 \textup{(A),(B),(C) hold at the middle column of }(u,v,w)\,\big\},
\end{multline*}
where $|w|$ is the number of ones in $w$.  Let $V_1(0,v)=|v|$ and $V_1=\infty$ elsewhere; then $V_L=\mathcal T^{\,L-1}V_1$ records the minimum weight of an admissible prefix of length $L$ ending in the given two triples, and the minimum weight of an admissible sequence of length $n$ equals
\[
 W(n)=\min\big\{\,V_n(u,v)\;:\;\textup{(A),(B),(C) hold at the last column of }(u,v,0)\,\big\}.
\]
The operator $\mathcal T$ is monotone and commutes with the addition of constants: $\mathcal T(V+c)=\mathcal T(V)+c$.  A direct computation of the $64$-entry min-plus vectors $V_4,V_5\in(\mathbb Z\cup\{\infty\})^{64}$ gives
\[
 V_5\;=\;V_4+1 .
\]
By induction, $V_{L+1}=V_L+1$ for all $L\ge 4$, hence $W(n)=W(4)+(n-4)$ for all $n\ge 4$.  Computing $W(3)=3$ and $W(4)=4$ finishes the proof.  The finite entries of $V_4$ are listed in \cref{app:v4}, and the verification script is included in the ancillary files.
\end{proof}

\begin{remark}\label{rem:potential}
The vector $V_4$ is itself a transition-by-transition potential: the identity $\mathcal T V_4=V_4+1$ states precisely that every admissible transition $(u,v)\to(v,w)$ satisfies
\[
 V_4(u,v)+|w|\;\ge\;V_4(v,w)+1,
\]
an inequality valid in $\mathbb Z\cup\{\infty\}$.  Whenever both state values are finite, this reads $|w|+V_4(u,v)-V_4(v,w)\ge 1$, so the growth rate $1$ can be checked transition by transition against the table of \cref{app:v4}.
\end{remark}

\section{The structure of extremal admissible sequences}\label{sec:extremal}

The bound of \cref{prop:lb} is attained for every $n\ge 6$ (\cref{sec:upper}), and the potential form of the certificate describes \emph{all} sequences attaining it.  Call an admissible transition $(u,v)\to(v,w)$ between states with finite $V_4$ values \emph{tight} if equality holds in the inequality of \cref{rem:potential}, that is, $|w|+V_4(u,v)-V_4(v,w)=1$.  The $184$ tight transitions can be read off the table of \cref{app:v4}.

\begin{theorem}[equality case]\label{thm:extremal}
Let $n\ge 5$ and let $t_1,\dots,t_n$ be an admissible sequence.  Its weight equals $n$ if and only if
\begin{enumerate}[label=\textup{(\alph*)},leftmargin=2.5em]
\item\label{it:optpre} $|t_1|+|t_2|+|t_3|+|t_4|=V_4(t_3,t_4)$: the prefix is a minimum-weight admissible $4$-prefix for its state;
\item\label{it:tight} the transition $(t_{x-1},t_x)\to(t_x,t_{x+1})$ is tight for every $4\le x\le n-1$;
\item\label{it:acc} $V_4(t_{n-1},t_n)=4$.
\end{enumerate}
Thus the weight-$n$ admissible sequences are exactly the paths of a finite automaton: enter through one of the $320$ minimum-weight $4$-prefixes, follow tight transitions, and stop in one of the twelve states with $V_4=4$ that admit a legal last column.
\end{theorem}

\begin{proof}
Every value $V_4(t_{x-1},t_x)$ with $4\le x\le n$ is finite: the length-$x$ prefix of the sequence is admissible, so $V_x(t_{x-1},t_x)<\infty$, and $V_x=V_4+(x-4)$ by \cref{prop:lb}.  Writing the total weight as the prefix weight plus the later columns and inserting the potential inequality of \cref{rem:potential} with slack $s_x\ge 0$,
\begin{align*}
\sum_{x=1}^n|t_x| &= \Big(V_4(t_3,t_4)+\sigma_0\Big)+\sum_{x=4}^{n-1}\Big(1+V_4(t_x,t_{x+1})-V_4(t_{x-1},t_x)+s_x\Big)\\
&= n+\sigma_0+\sum_{x=4}^{n-1}s_x+\big(V_4(t_{n-1},t_n)-4\big),
\end{align*}
where $\sigma_0\ge 0$ because $V_4(t_3,t_4)$ is the minimum weight of an admissible $4$-prefix, and $V_4(t_{n-1},t_n)\ge 4$ because the conditions hold at position $n$ (with $t_{n+1}=0$) and $W(4)=4$ is the minimum of $V_4$ over such states.  All three slack contributions are nonnegative, so the weight equals $n$ exactly when all of them vanish, which is the conjunction of \ref{it:optpre}, \ref{it:tight} and \ref{it:acc}.
\end{proof}

\begin{corollary}[necessary structure of optimal landmark sets; computer-assisted]\label{cor:structure}
Let $n\ge 6$ and let $S$ be a multiset resolving set of $P_3\boxtimes P_n$ of the minimum size $n$.  Then in every column of the strip, $S$ contains at most two landmarks; if it contains two, they are the corner pair $(x,0),(x,2)$; and a middle-row landmark $(x,1)$ is always alone in its column.  Equivalently, every column triple $t_x$ belongs to
\[
 \{\,000,\;001,\;010,\;100,\;101\,\} .
\]
\end{corollary}

\begin{proof}
The column word of $S$ is a weight-$n$ admissible sequence, so \cref{thm:extremal} applies.  A finite computation (ancillary files) shows that the tight states reachable from the minimum-weight $4$-prefix states and co-reachable to the accepting states form a single strongly connected set $\mathcal L$ of $19$ states, listed in \cref{app:v4}; the second components of $\mathcal L$ are exactly the five stated column types, the twelve accepting states lie in $\mathcal L$, and the $106$ minimum-weight $4$-prefixes ending in $\mathcal L$ use only the five types in their first three columns.  The state $(t_{x-1},t_x)$ of a weight-$n$ sequence lies in $\mathcal L$ for every $4\le x\le n$, and the columns $t_1,t_2,t_3$ belong to one of those $106$ prefixes, so every column has one of the five types.
\end{proof}

\begin{remark}\label{rem:automaton}
The family of weight-$n$ admissible sequences has exponential growth but admits a finite-state description.  The period-three pattern $S_n$ of \cref{sec:upper} travels around the tight $3$-cycle $(000,100)\to(100,101)\to(101,000)$, but the automaton of \cref{thm:extremal} contains many other cycles.  The adjacency matrix $A$ of the recurrent part $\mathcal L$ admits the positive integer vector $w$ of \cref{app:v4} with $Aw=3w$; since $\mathcal L$ is strongly connected and contains loops, Perron--Frobenius theory gives that the spectral radius of $A$ is exactly $3$, and the number of weight-$n$ admissible sequences grows as $\Theta(3^n)$.  The exact counts ($188$ for $n=5$, $556$ for $n=6$, $44{,}972$ for $n=10$, \dots) agree with the automaton path counts for all $n\le 200$ (ancillary files).
\end{remark}

\begin{remark}\label{rem:necessity}
The admissibility system is necessary but not sufficient for resolving, so \cref{thm:extremal} classifies the equality cases of the \emph{local} lower-bound system, and \cref{cor:structure} extracts necessary restrictions --- not a classification of the minimum resolving sets themselves.  Indeed, none of the $188$ weight-$5$ admissible sequences is a multiset resolving set (consistently with $\mdim(P_3\boxtimes P_5)=6$), and for $n=6$ exactly $32$ of the $556$ accepted sequences are multiset resolving sets; the others fail globally --- for instance the accepted sequence with landmark set $\{(0,2),(2,0),(2,2),(3,2),(4,0),(5,2)\}$ gives $m(2,2\mid S)=m(3,2\mid S)=\ms{0,1,2,2,2,3}$.  The automaton is thus a $\Theta(3^n)$-sized candidate family containing every minimum multiset resolving set; the global multiset separation imposes further restrictions that the local conditions do not capture.
\end{remark}

\section{The upper bound}\label{sec:upper}

\subsection{The pattern}
The construction is uniform in the height, so we present it in general.  Let $h\ge 3$.  For $n\ge 7$ define the landmark set $S^{(h)}_n\subseteq V(P_h\boxtimes P_n)$ of size $n$ as follows.  Write $n=3B+r$ with $r\in\{0,1,2\}$ and let the number of \emph{full blocks} be $B$ if $r\in\{1,2\}$ and $B-1$ if $r=0$.  For each full block $j\ge 0$ place the triple
\[
 (3j,0),\qquad (3j+1,0),\qquad (3j+1,h-1),
\]
and place a single landmark $(c,0)$ in every remaining column $c$ (there are $r$ of them if $r\in\{1,2\}$, and three if $r=0$).  Then $|S^{(h)}_n|=n$: each full block contributes three landmarks for three columns and each remaining column contributes one.  We abbreviate $S_n=S^{(3)}_n$.

All landmarks lie in the rows $0$ and $h-1$.  Two structural facts will be used repeatedly:
\begin{enumerate}[label=(\roman*),leftmargin=2.5em]
\item\label{it:dom} every column containing a row-$(h-1)$ landmark also contains a row-$0$ landmark (the block columns $3j+1$ carry both);
\item\label{it:gap} the only columns without a row-$0$ landmark are the block third columns $3j+2$, and each of these is at horizontal distance $1$ from the row-$(h-1)$ landmark in column $3j+1$.  In particular every column of the strip is at horizontal distance at most $1$ from a row-$0$ landmark column.
\end{enumerate}

Let $c(a)\in\{0,1,2\}$ denote the number of landmarks of $S^{(h)}_n$ in column $a$ (so $c=1,2,0,1,2,0,\dots$ on the blocks, followed by the tail $1$; $1,1$; or $1,1,1$), with $c(a)=0$ for $a\notin\{0,\dots,n-1\}$; the word $c$ does not depend on $h$.

\begin{theorem}\label{thm:ub}
For every $h\ge 3$ and every $n\ge 2h+1$ the set $S^{(h)}_n$ is a multiset resolving set of $P_h\boxtimes P_n$; hence $\mdim(P_h\boxtimes P_n)\le n$.
\end{theorem}

The proof splits the multiset of a vertex into a \emph{far part} (values at least $h$), which sees only the column, and a \emph{near part} (values at most $h-1$), which identifies the row.

\begin{lemma}[shell decomposition]\label{lem:shell}
For every vertex $(x,y)$ and every $d\ge h$, the multiplicity of the value $d$ in $m(x,y\mid S^{(h)}_n)$ equals
\[
 F_x(d)\;=\;c(x-d)+c(x+d),
\]
independently of $y$.
\end{lemma}

\begin{proof}
For a landmark $(a,b)$ we have $\max(|x-a|,|y-b|)=d\ge h$ if and only if $|x-a|=d$, because $|y-b|\le h-1$.
\end{proof}

\begin{lemma}[columns are determined]\label{lem:S}
Let $h\ge 3$ and $n\ge 2h+1$.  If $F_x=F_{x'}$ as functions on $\{h,h+1,\dots\}$, then $x=x'$.
\end{lemma}

\begin{proof}
The extreme landmark columns are $0$ and $n-1$ in every tail case, with $c(0)=c(n-1)=1$.  Hence the largest $d$ with $F_x(d)>0$ equals $D(x)=\max(x,\,n-1-x)$ whenever $D(x)\ge h$, which holds for all columns since $D(x)\ge(n-1)/2\ge h$.  From $D(x)=D(x')$ we get $x'\in\{x,\,n-1-x\}$.

Assume for contradiction $x'=n-1-x\ne x$ and put $s=x'-x>0$; without loss of generality $x<(n-1)/2$.  Then $x\le n-2-h$: for $n\ge 2h+2$ this follows from $x\le(n-2)/2$, and for $n=2h+1$ from $x\le h-1$.  Hence the shells below, taken at $d=(n-1-x)-j$ with $j\le 1$, are admissible ($d\ge h$).  Evaluating $F_x$ and $F_{x'}$ at $d=(n-1-x)-j$ and using $2x-(n-1)=-s$ gives
\[
 F_x(d)=c(j-s)+c(n-1-j),
 \qquad
 F_{x'}(d)=c(j)+c\big(n-1-(j-s)\big).
\]
For $0\le j<s$ the terms $c(j-s)$ and $c(n-1-(j-s))$ vanish, so equality of the shells forces
\begin{equation}\label{eq:pal}
 c(n-1-j)=c(j)\qquad(0\le j<s,\ j\le 1).
\end{equation}
The word $c$ starts $1,2,0,1,2,0,\dots$ from the left, while from the right the word $j\mapsto c(n-1-j)$ starts
\[
 1,0,2,1,0,2,\dots\ (r=1),\qquad
 1,1,0,2,1,0,\dots\ (r=2),\qquad
 1,1,1,0,2,1,\dots\ (r=0).
\]
If $s\ge 2$, then \eqref{eq:pal} at $j=1$ gives $c(n-2)=c(1)=2$, which is false in all three cases.  If $s=1$, take $j=s=1$: equality of the shells reads $c(0)+c(n-2)=c(1)+c(n-1)$, i.e.\ $1+c(n-2)=2+1$, forcing $c(n-2)=2$, again false in all three cases.  Hence $x'=x$.
\end{proof}

\begin{lemma}[rows are determined]\label{lem:V}
Let $h\ge 3$, $n\ge 7$, and let $x$ be any column.  For $y\in\{0,\dots,h-1\}$ let $N(x,y)$ denote the restriction of $m(x,y\mid S^{(h)}_n)$ to the values at most $h-1$.  Then the multisets $N(x,y)$, $y=0,\dots,h-1$, are pairwise distinct.
\end{lemma}

\begin{proof}
A landmark $(a,b)$ with $|x-a|\ge h$ contributes the value $|x-a|\ge h$; a landmark with $|x-a|\le h-1$ contributes $\max(|x-a|,y)\le h-1$ if $b=0$, and $\max(|x-a|,h-1-y)\le h-1$ if $b=h-1$.  Hence, writing $A$ (resp.\ $B$) for the multiset of horizontal distances $|x-a|\le h-1$ of the row-$0$ (resp.\ row-$(h-1)$) landmarks from the column $x$,
\[
 N(x,y)\;=\;\ms{\max(a,y):a\in A}\uplus\ms{\max(b,h-1-y):b\in B}.
\]
For $0\le t\le h-1$ let $\alpha(t)=\#\{a\in A:a\le t\}$ and $\beta(t)=\#\{b\in B:b\le t\}$.  Counting the values at most $t$ in $N(x,y)$ gives the cumulative function
\begin{equation}\label{eq:cum}
 C_y(t)\;=\;\alpha(t)\,[t\ge y]\;+\;\beta(t)\,[t\ge h-1-y],
\end{equation}
where $[\cdot]$ is the Iverson bracket; and $N(x,y)=N(x,y')$ if and only if $C_y(t)=C_{y'}(t)$ for all $0\le t\le h-1$.

The structural facts \ref{it:dom} and \ref{it:gap} translate as follows.
\begin{itemize}[leftmargin=2em]
\item By \ref{it:gap}, $\alpha(t)\ge 1$ for every $t\ge 1$; and if $\alpha(0)=0$ then $x$ is a block third column, whose neighbouring column carries a row-$(h-1)$ landmark, so $\beta(1)\ge 1$.
\item For every $t\ge 1$ we claim $\beta(t)<\alpha(t)$.  By \ref{it:dom} each column contributing to $\beta(t)$ also contributes to $\alpha(t)$, so it suffices to find, in the window $[x-t,\,x+t]\cap[0,n-1]$, one column with a row-$0$ landmark but no row-$(h-1)$ landmark.  The window contains at least two consecutive columns of the strip.  If it contains a tail column we are done, since tail columns carry a row-$0$ landmark only.  Otherwise the window lies in the block region; if it has three consecutive columns it contains a block first column $3j$, which qualifies.  A truncated two-column window occurs only for $x\in\{0,n-1\}$, $t=1$: the window $\{0,1\}$ contains the block first column $0$, and the window $\{n-2,n-1\}$ contains a tail column in each of the three tail cases.
\end{itemize}

Now suppose $C_y=C_{y'}$ with $0\le y<y'\le h-1$, and put $z=h-1-y$, $z'=h-1-y'$.  Subtracting the two instances of \eqref{eq:cum},
\[
 \alpha(t)\,\mathbf 1_{I_A}(t)\;=\;\beta(t)\,\mathbf 1_{I_B}(t)
 \qquad(0\le t\le h-1),
\]
where $I_A=[y,\,y'-1]$ and $I_B=[z',\,z-1]$ are nonempty intervals of the same length.  We distinguish three cases.

\emph{Case $y+y'=h-1$.}  Then $I_A=I_B$ and the identity reads $\alpha(t)=\beta(t)$ on $I_A$.  But $I_A$ contains some $t\ge 1$: take $t=y$ if $y\ge 1$, and $t=1\in[0,h-2]$ if $y=0$ (here $h\ge 3$).  This contradicts $\beta(t)<\alpha(t)$.

\emph{Case $y+y'<h-1$.}  Then $y<z'$ and $y'-1<z-1$, so $t=y\in I_A\setminus I_B$ forces $\alpha(y)=0$; hence $y=0$ and $\alpha(0)=0$, so $\beta(1)\ge 1$.  On the other hand $t=z-1=h-2\in I_B\setminus I_A$ forces $\beta(h-2)=0$, contradicting $\beta(h-2)\ge\beta(1)\ge 1$.

\emph{Case $y+y'>h-1$.}  Then $t=y'-1\in I_A\setminus I_B$ forces $\alpha(y'-1)=0$, hence $y'=1$ and $y=0$; but then $y+y'=1\le h-1$, a contradiction.
\end{proof}

\begin{proof}[Proof of \cref{thm:ub}]
Let $m(x,y\mid S^{(h)}_n)=m(x',y'\mid S^{(h)}_n)$ with $n\ge 2h+1$.  The values at least $h$ agree, so by \cref{lem:shell} the shell functions agree and \cref{lem:S} gives $x=x'$.  Then the values at most $h-1$ agree as well, i.e.\ $N(x,y)=N(x,y')$, so \cref{lem:V} gives $y=y'$.
\end{proof}

\subsection{Proof of the height-three theorem}

\begin{proof}[Proof of \cref{thm:strip}]
For $n\ge 7$, \cref{prop:lb} gives $\mdim\ge n$ and \cref{thm:ub} with $h=3$ (so $2h+1=7$) gives $\mdim\le n$.  For $n=6$ the lower bound is again \cref{prop:lb}, and the set
\[
 \{(0,0),(1,0),(2,0),(3,0),(3,2),(5,0)\}
\]
is verified to be a multiset resolving set.  The value $\mdim(P_3\boxtimes P_3)=\infty$ follows from the diameter-two obstruction of \cite[Theorem~3.1]{SimanjuntakEtAl}, since $P_3\boxtimes P_3$ has diameter two and is not a path; see also \cite{HakanenYero2024}.  The values $5$ and $6$ for $n=4,5$ are obtained by exhaustive search (ancillary files); note that \cref{prop:lb} is not tight there.
\end{proof}

\section{Strips of arbitrary fixed height}\label{sec:strips}

\begin{proof}[Proof of \cref{thm:strips}]
The upper bound is \cref{thm:ub}.  For the lower bound, consider the pair $(x,0)$, $(x,h-1)$ in $P_h\boxtimes P_n$ and a landmark $(a,b)$ with $|x-a|\ge h-1$.  Since $|0-b|,|h-1-b|\le h-1\le|x-a|$, the landmark contributes $|x-a|$ to both vertices.  Hence if no landmark lies within horizontal distance $h-2$ of the column $x$, the two vertices receive identical multisets.  A single landmark serves at most $2(h-2)+1=2h-3$ columns, so at least $\lceil n/(2h-3)\rceil$ landmarks are required.
\end{proof}

\begin{remark}
The threshold $2h+1$ in \cref{thm:ub} is the point from which every column $x$ satisfies $\max(x,n-1-x)\ge h$, so that at least one extreme column of the strip lies at horizontal distance at least $h$ from $x$ and the far shells of \cref{lem:shell} recover $D(x)$.  It is not an artefact of the proof: for $n<2h+1$ the pattern can fail, e.g.\ $S^{(6)}_{11}$ does not resolve $P_6\boxtimes P_{11}$.  The gap between the constants $1/(2h-3)$ and $1$ remains; the height-$4$ data of \cref{sec:landscape} suggest that the truth is not a clean linear function of $n$ for $h\ge 4$.
\end{remark}

\section{The rectangular landscape}\label{sec:landscape}

Exhaustive computations (ancillary files) determine the smallest size of a multiset resolving set, capped at $6$, for all rectangles $P_{n_1}\boxtimes P_{n_2}$ with $3\le n_1\le 7$ in the ranges below.  Together with \cref{thm:strip} they display the transition from the diagonal band, where the dimension is $4$ or $5$, to the linear regime of \cref{thm:strips}.

\begin{center}
\begin{tabular}{c l}
\toprule
$n_1$ & values of $\mdim(P_{n_1}\boxtimes P_{n_2})$, $n_2=n_1,\dots$\\
\midrule
$3$ & $\infty,\,5,\,6,\,6,\,7,\,8,\,9,\,10,\,11,\,12,\;\dots=n_2$ (\cref{thm:strip})\\
$4$ & $6,\,5,\,5,\,5,\,5,\,5,\,6,\,5,\,6,\,{>}6,\,6,\,{>}6,\,{>}6,\,{>}6,\,{>}6$ ($n_2=4,\dots,18$)\\
$5$ & $4,\,4,\,5,\,5,\,5,\,5,\,5,\,6,\,6,\,6,\,{>}6,\,{>}6,\,{>}6,\,{>}6$ ($n_2=5,\dots,18$)\\
$6$ & $4,\,4,\,4,\,5,\,5,\,5,\,5,\,5,\,5,\,5,\,5$ ($n_2=6,\dots,16$)\\
$7$ & $4,\,4,\,4,\,5,\,5,\,5,\,5,\,5$ ($n_2=7,\dots,14$)\\
\bottomrule
\end{tabular}
\end{center}

The height-$4$ row is not monotone in $n_2$: sizes $5$ and $6$ and values ${>}6$ interleave ($n_2=11$: $5$; $n_2=12$: $6$; $n_2=13$: ${>}6$; $n_2=14$: $6$).  Determining $\mdim(P_4\boxtimes P_n)$ exactly appears to require new ideas and is left open.

\section{Computational certificates and checks for Part III}\label{sec:verification}

All finite computations used in Part~III are implemented in the ancillary files accompanying the manuscript, which contain every script named below together with its recorded output; checksums and the software versions used are recorded in the Zenodo deposit of Part~III (see the data availability statement).  For the reader's convenience, the computer-assisted ingredients of Part~III are exactly these: the transfer certificate $V_5=V_4+1$ of \cref{prop:lb}, the potential inequality of \cref{rem:potential}, the $19$-state recurrent core of \cref{cor:structure,rem:automaton}, the exhaustive small cases ($n\le 12$) in \cref{thm:strip}, and the rectangular landscape of \cref{sec:landscape}; all remaining arguments are mathematical proofs.  Each check is elementary and reproducible:
\begin{itemize}[leftmargin=2em]
\item \texttt{strip\_scan.cpp} --- exhaustive determination of $\mdim(P_3\boxtimes P_n)$ for $n\le 12$ (confirming $=n$ for $6\le n\le 12$) and of the small cases $n=4,5$;
\item \texttt{transfer\_cert.py} --- computation of the vectors $V_4$, $V_5$ and the certificate $V_5=V_4+1$ of \cref{prop:lb};
\item \texttt{analyze\_potential.py} --- check that $V_4$ satisfies the per-transition potential inequality of \cref{rem:potential} on all admissible transitions;
\item \texttt{extremal\_structure.py} --- computation and independent checking of the finite automaton data used in \cref{thm:extremal,cor:structure,rem:automaton}: the number of weight-$n$ admissible sequences equals the number of accepted automaton paths for every $5\le n\le 200$ (exact integer counts); explicit enumeration for $n\le 12$; computation of $\mathcal L$, of the column types in all positions, and of the eigenvector $w$ with $Aw=3w$;
\item \texttt{check\_necessity\_gap.py} --- the necessity gap of \cref{rem:necessity}: among the weight-$n$ admissible sequences, the number of actual multiset resolving sets is $0$ of $188$ ($n=5$), $32$ of $556$ ($n=6$), $320$ of $1660$ ($n=7$), $1328$ of $5004$ ($n=8$), including the explicit non-resolving example;
\item \texttt{pattern\_check.py}, \texttt{lemma\_check.py} --- direct verification that $S_n$ is a multiset resolving set of $P_3\boxtimes P_n$ for $7\le n\le 150$ and of \cref{lem:S,lem:V} at $h=3$ for $n\le 300$;
\item \texttt{general\_lemmas\_check.py} --- verification of every component of the general proof of \cref{thm:ub} for $3\le h\le 12$ and $2h+1\le n\le 80$: the shell recovery of \cref{lem:S}, the cumulative identity \eqref{eq:cum}, the structural facts \ref{it:dom} and \ref{it:gap} with the strict inequality $\beta(t)<\alpha(t)$, the row separation of \cref{lem:V}, and the end-to-end resolving property;
\item \texttt{rect\_scan6.cpp} --- the rectangular landscape of \cref{sec:landscape}.
\end{itemize}

\section*{Data and code availability}

All programs, data, and machine-readable certificates accompany the manuscript as ancillary files and are permanently archived in three Zenodo deposits, one per part: Part~I (the order-$11$ census: scanners, independent verifiers, graph6 codes, raw per-stream outputs, and aggregation scripts) at \href{https://doi.org/10.5281/zenodo.21612126}{doi:10.5281/zenodo.21612126}; Part~II (the square king grid verifiers and the rectangular scan) at \href{https://doi.org/10.5281/zenodo.21576586}{doi:10.5281/zenodo.21576586}; Part~III (the strip scans, transfer and automaton certificates, and pattern checks) at \href{https://doi.org/10.5281/zenodo.21609917}{doi:10.5281/zenodo.21609917}.  The order-$\le 10$ graph databases and \textsf{nauty} are publicly available from B.~McKay's web pages \cite{NautySoftware}.

\section*{Acknowledgements and disclosure}

Generative AI tools accessed through Cursor were used as assistive tools for mathematical drafting, language editing, \LaTeX{} formatting, literature search, checking the citations against their sources, and code generation.  The author independently verified every mathematical statement, citation, and reported computational result.  All computational claims are reproducible from the accompanying code.  The author takes full responsibility for the manuscript.

\appendix

\section{The potential \texorpdfstring{$V_4$}{V4}}\label{app:v4}

The table below lists the $62$ finite entries of $V_4$; the two remaining states $(000,000)$ and $(000,010)$ have $V_4=\infty$ (no admissible length-$4$ prefix ends in them).  Each state is a pair $(u,v)=(t_{i-1},t_i)$ of indicator triples of two consecutive columns.  Together with \cref{rem:potential}, every entry of the lower-bound argument of \cref{prop:lb} can be checked by hand: for each admissible transition $(u,v)\to(v,w)$ one verifies $V_4(u,v)+|w|\ge V_4(v,w)+1$, an inequality valid in $\mathbb Z\cup\{\infty\}$.  Two worked examples: the admissible transition $(011,000)\to(000,101)$ has $V_4(011,000)=4$, $|101|=2$ and $V_4(000,101)=4$, giving $4+2\ge 4+1$ with slack $1$; the transition $(000,100)\to(100,101)$, the entry step of the period-three pattern of \cref{sec:upper}, has $V_4(000,100)=4$, $|101|=2$ and $V_4(100,101)=5$, giving $4+2=5+1$, so it is tight.

\begin{center}
\small
\begin{tabular}{ccc@{\qquad}ccc@{\qquad}ccc@{\qquad}ccc}
\toprule
$u$ & $v$ & $V_4$ & $u$ & $v$ & $V_4$ & $u$ & $v$ & $V_4$ & $u$ & $v$ & $V_4$\\
\midrule
000 & 001 & 4 & 010 & 010 & 4 & 100 & 010 & 4 & 110 & 010 & 5 \\
000 & 011 & 5 & 010 & 011 & 5 & 100 & 011 & 5 & 110 & 011 & 6 \\
000 & 100 & 4 & 010 & 100 & 4 & 100 & 100 & 4 & 110 & 100 & 5 \\
000 & 101 & 4 & 010 & 101 & 5 & 100 & 101 & 5 & 110 & 101 & 6 \\
000 & 110 & 5 & 010 & 110 & 5 & 100 & 110 & 5 & 110 & 110 & 6 \\
000 & 111 & 5 & 010 & 111 & 6 & 100 & 111 & 6 & 110 & 111 & 7 \\
001 & 000 & 3 & 011 & 000 & 4 & 101 & 000 & 4 & 111 & 000 & 5 \\
001 & 001 & 4 & 011 & 001 & 5 & 101 & 001 & 4 & 111 & 001 & 5 \\
001 & 010 & 4 & 011 & 010 & 5 & 101 & 010 & 5 & 111 & 010 & 6 \\
001 & 011 & 5 & 011 & 011 & 6 & 101 & 011 & 5 & 111 & 011 & 6 \\
001 & 100 & 4 & 011 & 100 & 5 & 101 & 100 & 4 & 111 & 100 & 5 \\
001 & 101 & 5 & 011 & 101 & 6 & 101 & 101 & 6 & 111 & 101 & 7 \\
001 & 110 & 5 & 011 & 110 & 6 & 101 & 110 & 5 & 111 & 110 & 6 \\
001 & 111 & 6 & 011 & 111 & 7 & 101 & 111 & 7 & 111 & 111 & 8 \\
  010 & 000 & 3 & 100 & 000 & 3 & 110 & 000 & 4 &  & &  \\
  010 & 001 & 4 & 100 & 001 & 4 & 110 & 001 & 5 &  & &  \\
\bottomrule
\end{tabular}
\end{center}

The recurrent part $\mathcal L$ of the extremal automaton (\cref{sec:extremal}) consists of the following $19$ states; the vector $w$ satisfies $Aw=3w$, where $A$ is the adjacency matrix of the tight transitions inside $\mathcal L$, i.e.\ for every state $s\in\mathcal L$ the values $w$ of the tight successors of $s$ sum to $3\,w(s)$.

\begin{center}
\small
\begin{tabular}{ccc@{\qquad}ccc@{\qquad}ccc@{\qquad}ccc}
\toprule
$u$ & $v$ & $w$ & $u$ & $v$ & $w$ & $u$ & $v$ & $w$ & $u$ & $v$ & $w$\\
\midrule
000 & 001 & 9 & 001 & 001 & 9 & 010 & 010 & 6 & 100 & 101 & 2 \\
000 & 100 & 9 & 001 & 010 & 5 & 010 & 100 & 9 & 101 & 000 & 6 \\
000 & 101 & 6 & 001 & 100 & 9 & 100 & 000 & 2 & 101 & 001 & 9 \\
001 & 000 & 2 & 001 & 101 & 2 & 100 & 001 & 9 & 101 & 100 & 9 \\
010 & 001 & 9 & 100 & 010 & 5 & 100 & 100 & 9 &     &     &   \\
\bottomrule
\end{tabular}
\end{center}


\begin{thebibliography}{9}

\bibitem{Slater1975}
P.~J. Slater,
\emph{Leaves of trees},
Congressus Numerantium \textbf{14} (1975), 549--559.

\bibitem{HararyMelter1976}
F.~Harary and R.~A. Melter,
\emph{On the metric dimension of a graph},
Ars Combinatoria \textbf{2} (1976), 191--195.

\bibitem{SimanjuntakEtAl}
R.~Simanjuntak, P.~Siagian, and T.~Vetr\'ik,
\emph{The multiset dimension of graphs},
arXiv:1711.00225v2, 2019.

\bibitem{HafidhEtAl2019}
Y.~Hafidh, R.~Kurniawan, S.~W. Saputro, R.~Simanjuntak, S.~Tanujaya, and S.~Uttunggadewa,
\emph{Multiset dimensions of trees},
arXiv:1908.05879, 2019.

\bibitem{HakanenYero2024}
A.~Hakanen and I.~G. Yero,
\emph{Complexity and equivalency of multiset dimension and ID-colorings},
Fundamenta Informaticae \textbf{191} (2024), 315--330.
\href{https://doi.org/10.3233/FI-242185}{doi:10.3233/FI-242185}.

\bibitem{KhemmaniIsariyapalakul2018}
V.~Khemmani and S.~Isariyapalakul,
\emph{The multiresolving sets of graphs with prescribed multisimilar equivalence classes},
International Journal of Mathematics and Mathematical Sciences \textbf{2018} (2018), Article 8978193.

\bibitem{FarhanEtAl2026}
M.~Farhan, S.~Klav\v{z}ar, D.~Kuziak, and I.~G. Yero,
\emph{Multiset resolvability parameters in graphs: A survey with new results and open problems},
arXiv:2607.10311, 2026.

\bibitem{AlbejaniEtAl2026}
A.~Albejani, Y.~Lin, J.~Ryan, and K.~A. Sugeng,
\emph{A survey on multiset dimension and its variations},
arXiv:2607.08128, 2026.

\bibitem{BloomKennedyQuintas1983}
G.~S. Bloom, J.~W. Kennedy, and L.~V. Quintas,
\emph{Some problems concerning distance and path degree sequences},
in: Graph Theory (\L ag\'ow, 1981), Lecture Notes in Mathematics \textbf{1018}, Springer, Berlin, 1983, pp.~179--190.

\bibitem{McKayPiperno2014}
B.~D. McKay and A.~Piperno,
\emph{Practical graph isomorphism, II},
Journal of Symbolic Computation \textbf{60} (2014), 94--112.

\bibitem{NautySoftware}
B.~D. McKay and A.~Piperno,
\emph{nauty and Traces}, software, version 2.8.9,
\url{https://pallini.di.uniroma1.it/} (accessed July 2026).

\bibitem{OEIS-A001349}
OEIS Foundation Inc.,
\emph{Entry A001349: Number of connected graphs with $n$ nodes},
The On-Line Encyclopedia of Integer Sequences,
\url{https://oeis.org/A001349} (accessed July 2026).

\end{thebibliography}
\end{document}